\documentclass{article}%
\usepackage{amssymb}
\usepackage{amsmath}
\usepackage{amsfonts}
\usepackage{graphicx}%
\providecommand{\U}[1]{\protect\rule{.1in}{.1in}}
\newtheorem{theorem}{Theorem}[section]

\newtheorem{corollary}[theorem]{Corollary}

\newtheorem{definition}[theorem]{Definition}
\newtheorem{example}[theorem]{Example}

\newtheorem{lemma}[theorem]{Lemma}

\newtheorem{problem}[theorem]{Problem}
\newtheorem{proposition}[theorem]{Proposition}
\newtheorem{remark}[theorem]{Remark}

\newenvironment{proof}[1][Proof]{\noindent\textbf{#1.} }{\ \rule{0.5em}{0.5em}}
\begin{document}

\author{Vadim E. Levit\\Department of Mathematics\\Ariel University, Israel\\levitv@ariel.ac.il
\and Eugen Mandrescu\\Department of Computer Science\\Holon Institute of Technology, Israel\\eugen\_m@hit.ac.il}
\title{On $1$-K\"{o}nig-Egerv\'{a}ry Graphs}
\date{}
\maketitle

\begin{abstract}
Let $\alpha(G)$ denote the cardinality of a maximum independent set, while
$\mu(G)$ be the size of a maximum matching in $G=\left(  V,E\right)  $. Let
$\xi(G)$ denote the size of the intersection of all maximum independent sets
\cite{LevMan2002a}. It is known that if $\alpha(G)+\mu(G)=n(G)=\left\vert
V\right\vert $, then $G$ is a \textit{K\"{o}nig-Egerv\'{a}ry graph
}\cite{dem,ster}. If $\alpha(G)+\mu(G)=n(G)-1$, then $G$ is a $1$%
\textit{-K\"{o}nig-Egerv\'{a}ry graph}. If $G$ is not a K\"{o}nig-Egerv\'{a}ry
graph, and there exists a vertex $v\in V$ (an edge $e\in E$) such that $G-v$
($G-e$) is K\"{o}nig-Egerv\'{a}ry, then $G$ is called a vertex (an edge)
almost K\"{o}nig-Egerv\'{a}ry graph (respectively). The \textit{critical
difference} $d(G)$ is $\max\{d(I):I\in\mathrm{Ind}(G)\}$, where $\mathrm{Ind}%
(G)$\ denotes the family of all independent sets of $G$. If $A\in
\mathrm{Ind}(G)$ with $d\left(  X\right)  =d(G)$, then $A$ is a
\textit{critical independent set} \cite{Zhang1990}. Let $\emph{diadem}%
(G)=\bigcup\{S:S$ is a critical independent set in $G\}$ \cite{JarLevMan2019},
and $\varrho_{v}\left(  G\right)  $ denote the number of vertices $v\in
V\left(  G\right)  $, such that $G-v$ is a K\"{o}nig-Egerv\'{a}ry graph
\cite{LevMan2024}.

In this paper, we characterize all types of almost K\"{o}nig-Egerv\'{a}ry
graphs and present interrelationships between them. We also show that if $G$
is a $1$-K\"{o}nig-Egerv\'{a}ry graph, then $\varrho_{v}\left(  G\right)  \leq
n\left(  G\right)  +d\left(  G\right)  -\xi\left(  G\right)  -\beta(G)$, where
$\beta(G)=\left\vert \emph{diadem}(G)\right\vert $. As an
application,{\LARGE ~}we characterize the $1$-K\"{o}nig-Egerv\'{a}ry graphs
that become K\"{o}nig-Egerv\'{a}ry after deleting any vertex.


\end{abstract}

\section{Introduction}

Throughout this paper $G=(V,E)$ is a finite, undirected, loopless graph
without multiple edges, with vertex set $V=V(G)$ of cardinality $\left\vert
V\left(  G\right)  \right\vert =n\left(  G\right)  $, and edge set $E=E(G)$ of
size $\left\vert E\left(  G\right)  \right\vert =m\left(  G\right)  $.

If $X\subset V$, then $G[X]$ is the subgraph of $G$ induced by $X$. By $G-v$
we mean the subgraph $G[V-\left\{  v\right\}  ]$, for $v\in V$. The
\textit{neighborhood} of a vertex $v\in V$ is the set $N(v)=\{w:w\in V$ and
$vw\in E\}$. The \textit{neighborhood} of $A\subseteq V$ is $N(A)=\{v\in
V:N(v)\cap A\neq\emptyset\}$, and $N\left[  A\right]  =A\cup N(A)$, or
$N_{G}(A)$ and $N_{G}\left[  A\right]  $, if we specify the graph. If
$A,B\subset V$ are disjoint, then $\left(  A,B\right)  =\left\{  ab:ab\in
E,a\in A,b\in B\right\}  $.

A set $S\subseteq V(G)$ is \textit{independent} if no two vertices from $S$
are adjacent, and by $\mathrm{Ind}(G)$ we mean the family of all the
independent sets of $G$. An independent set of maximum size is a
\textit{maximum independent set} of $G$, and $\alpha(G)=\max\{\left\vert
S\right\vert :S\in\mathrm{Ind}(G)\}$. Let $\Omega(G)$ be the family of all
maximum independent sets, and $\mathrm{core}(G)=\bigcap\{S:S\in\Omega(G)\}$,
while $\xi(G)=\left\vert \mathrm{core}(G)\right\vert $ \cite{LevMan2002a}. A
vertex $v\in V(G)$ is $\alpha$-\textit{critical }provided $\alpha
(G-v)<\alpha(G)$. Clearly, $\mathrm{core}(G)$ is the set of all $\alpha
$-critical vertices of $G$. An edge $e\in E(G)$ is $\alpha$-\textit{critical
}provided $\alpha(G)<\alpha(G-e)$. Notice that there are graphs in which every
edge is $\alpha$-critical (e.g., all $C_{2k+1}$ for $k\geq1$) or no edge is
$\alpha$-critical (e.g., all $C_{2k}$ for $k\geq2$).

For $X\subseteq V(G)$, the number $\left\vert X\right\vert -\left\vert
N(X)\right\vert $ is the \textit{difference} of $X$, denoted $d(X)$. The
\textit{critical difference} $d(G)$ is $\max\{d(I):I\in\mathrm{Ind}(G)\}$. If
$A\in\mathrm{Ind}(G)$ with $d\left(  X\right)  =d(G)$, then $A$ is a
\textit{critical independent set} \cite{Zhang1990}. For a graph $G$, let
$\mathrm{MaxCritIndep}(G)=\{S:S$ \textit{is a maximum critical independent
set}$\}$ \cite{LevMan2022c}, $\mathrm{\ker}(G)$ be the intersection of all its
critical independent sets and $\varepsilon(G)=|\mathrm{\ker}(G)|$
\cite{LevMan2012a,LevMan2012c}. The \textit{critical independence number} of a
graph $G$, denoted as $\alpha^{\prime}\left(  G\right)  $, is the cardinality
of a maximum critical independent set \cite{Larson2011}.

\begin{theorem}
\label{th333}\emph{(i)} \cite{ButTruk2007} Each critical independent set is
included in some $S\in\Omega(G)$.

\emph{(ii)} \cite{Larson2007} Every critical independent set is contained in
some $S\in\mathrm{MaxCritIndep}(G)$.

\emph{(iii)} \cite{Larson2007} There is a matching from $N(S)$ into $S$ for
every critical independent set $S$.\bigskip
\end{theorem}

\begin{theorem}
\label{th100}\cite{Larson2011} For any graph $G$, there is a unique set
$X\subseteq V(G)$ such that

\emph{(i)} $\alpha(G)=\alpha\left(  G\left[  X\right]  \right)  +\alpha\left(
G\left[  V\left(  G\right)  -X\right]  \right)  $;

\emph{(ii)} $X=N\left[  A\right]  $\ for every $A\in\mathrm{MaxCritIndep}(G)$; \ 

\emph{(iii)} $G\left[  X\right]  $ is a \textit{K\"{o}nig-Egerv\'{a}ry} graph.

\emph{(iv) }$G\left[  V\left(  G\right)  -X\right]  $ has only $\emptyset$ as
a critical independent set.
\end{theorem}

Let $\emph{diadem}(G)=\bigcup\{S:S$ is a critical independent set in $G\}$
\cite{JarLevMan2019}, while $\beta(G)=\left\vert \emph{diadem}(G)\right\vert $.

\begin{theorem}
\label{th444}\cite{LevMan2012a} If $A$ and $B$ are critical sets in a graph
$G$, then $A\cup B$ and $A\cap B$ are critical as well.
\end{theorem}

A \textit{matching} in a graph $G=(V,E)$ is a set of edges $M\subseteq E$ such
that no two edges of $M$ share a common vertex. A matching of maximum
cardinality $\mu(G)$ is a \textit{maximum matching}, and a \textit{perfect
matching} is one saturating all vertices of $G$. Given a matching $M$ in $G$,
a vertex $v\in V$ is called $M$-saturated if there exists an edge $e\in M$
incident with $v$.

An edge $e\in E(G)$ is $\mu$-\textit{critical }provided $\mu(G-e)<\mu(G)$. A
vertex $v\in V(G)$ is $\mu$-\textit{critical (essential) }provided
$\mu(G-v)<\mu(G)$, i.e., $v$ is $M$-saturated by every maximum matching $M$ of
$G$.

It is known that%
\[
\lfloor n\left(  G\right)  /2\rfloor+1\leq\alpha(G)+\mu(G)\leq n\left(
G\right)  \leq\alpha(G)+2\mu(G)
\]
hold for every graph $G$ \cite{BGL2002}. If $\alpha(G)+\mu(G)=n\left(
G\right)  $, then $G$ is called a K\"{o}nig-Egerv\'{a}ry graph\textit{
}\cite{dem,ster}.If $S$ is an independent set of a graph $G$ and $A=V\left(
G\right)  -S$, then we write $G=S\ast A$. For instance, if $E(G[A])=\emptyset
$, then $G=S\ast A$ is bipartite; if $G[A]$ is a complete graph, then $G=S\ast
A$ is a split graph.

\begin{theorem}
\label{th715}For a graph $G$, the following properties are equivalent:

\emph{(i)} $G$ is a \textit{K\"{o}nig-Egerv\'{a}ry graph};

\emph{(ii) }\cite{LevMan2002a} $G=S\ast A$, where $S\in$ $\mathrm{Ind}(G)$,
$\left\vert S\right\vert \geq\left\vert A\right\vert $, and $\left(
S,A\right)  $ contains a matching $M$ with $\left\vert M\right\vert
=\left\vert A\right\vert $;

\emph{(iii) }\cite{LevMan2013b} each maximum matching of $G$ matches $V\left(
G\right)  -S$ into $S$, for every $S\in\Omega(G)$;

\emph{(iv)} \cite{LevMan2012a} every $S\in%
\Omega
(G)$ is critical.
\end{theorem}

\begin{definition}
If $\alpha(G)+\mu(G)=n\left(  G\right)  -1$, then $G$ is called a
$1$-K\"{o}nig-Egerv\'{a}ry graph.
\end{definition}

For instance, both $G_{1}$ and $G_{2}$ from Figure \ref{fig123} are
$1$-K\"{o}nig-Egerv\'{a}ry graphs.

\begin{figure}[h]
\setlength{\unitlength}{1cm}\begin{picture}(5,1.2)\thicklines
\multiput(2,0)(1,0){4}{\circle*{0.29}}
\multiput(2,1)(1,0){3}{\circle*{0.29}}
\put(2,0){\line(1,0){3}}
\put(2,0){\line(0,1){1}}
\put(2,0){\line(1,1){1}}
\put(2,0){\line(2,1){2}}
\put(2,1){\line(1,-1){1}}
\put(2,1){\line(2,-1){2}}
\put(2,1){\line(3,-1){3}}
\put(3,0){\line(0,1){1}}
\put(3,0){\line(1,1){1}}
\put(3,1){\line(1,-1){1}}
\put(3,1){\line(2,-1){2}}
\put(4,0){\line(0,1){1}}
\put(4,1){\line(1,-1){1}}
\put(5.4,0){\makebox(0,0){$v_{1}$}}
\put(4.4,1){\makebox(0,0){$v_{2}$}}
\put(1,0.5){\makebox(0,0){$G_{1}$}}
\multiput(8,0)(1,0){5}{\circle*{0.29}}
\multiput(11,1)(1,0){2}{\circle*{0.29}}
\put(8,1){\circle*{0.29}}
\put(8,0){\line(1,0){4}}
\put(8,0){\line(0,1){1}}
\put(8,1){\line(1,-1){1}}
\put(11,1){\line(1,0){1}}
\put(10,0){\line(1,1){1}}
\put(12,0){\line(0,1){1}}
\put(9.5,0.2){\makebox(0,0){$e_{1}$}}
\put(10.55,0.2){\makebox(0,0){$e_{2}$}}
\put(7,0.5){\makebox(0,0){$G_{2}$}}
\end{picture}\caption{$G_{1}-v_{1}$, $G_{2}-e_{2}$ are K\"{o}nig-Egerv\'{a}ry
graphs, while $G_{1}-v_{2}$ and $G_{2}-e_{1}$ are not K\"{o}nig-Egerv\'{a}ry
graphs.}%
\label{fig123}%
\end{figure}
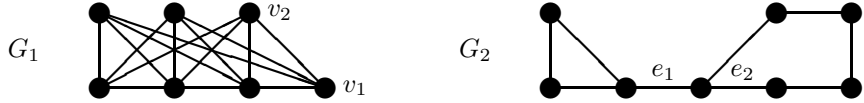

The \textit{K\"{o}nig deficiency} of graph $G$ is $\kappa\left(  G\right)
=n\left(  G\right)  -\left(  \alpha(G)+\mu(G)\right)  $ \cite{BartaKres2020}.
Thus, a graph $G$ is $1$-K\"{o}nig-Egerv\'{a}ry if and only if $\kappa\left(
G\right)  =1$. A subfamily of $1$-K\"{o}nig-Egerv\'{a}ry graphs was
characterized in \cite{DZ2021}, settling an open problem posed in
\cite{LevMan2019}.

\begin{definition}
A graph $G$ is called:

\emph{(i)} a vertex almost K\"{o}nig-Egerv\'{a}ry graph if $G$ is not
K\"{o}nig-Egerv\'{a}ry, but there is a vertex $v\in V(G)$, such that $G-v$ is
a K\"{o}nig-Egerv\'{a}ry graph \cite{LarsonPepper2011};

\emph{(ii)} an edge almost K\"{o}nig-Egerv\'{a}ry graph if $G$ is not a
K\"{o}nig-Egerv\'{a}ry graph, but there exists an edge $e\in E(G)$, such that
$G-e$ is a \textit{K\"{o}nig-Egerv\'{a}ry graph}.
\end{definition}

For example, the graph $G_{1}$ from Figure \ref{fig123} is vertex almost
K\"{o}nig-Egerv\'{a}ry but not edge almost K\"{o}nig-Egerv\'{a}ry, while
$G_{2}$ is edge almost K\"{o}nig-Egerv\'{a}ry but not vertex almost K\"{o}nig-Egerv\'{a}ry.

\begin{lemma}
\cite{LarsonPepper2011}\label{Larson lemma} A graph $G$ is a vertex almost
K\"{o}nig-Egerv\'{a}ry\textit{ graph} if and only if there is a vertex $v\in
V(G)$ such that $G-v$ is a \textit{K\"{o}nig-Egerv\'{a}ry graph},
$\alpha(G-v)=\alpha(G)$, and $%
\mu
(G-v)=%
\mu
(G)$.
\end{lemma}

Clearly, every odd cycle $C_{2k+1}$ is a vertex almost K\"{o}nig-Egerv\'{a}ry
graph, an edge almost K\"{o}nig-Egerv\'{a}ry graph and an almost
K\"{o}nig-Egerv\'{a}ry graph.

\begin{lemma}
\label{lem1}For a graph $G$, the following assertions hold:

\emph{(i)} $\alpha(G)\leq\alpha(G-e)\leq\alpha(G)+1$ for each $e\in E(G)$;

\emph{(ii)} $\alpha(G)-1\leq\alpha(G-v)\leq\alpha(G)$ for every $v\in V(G)$;

\emph{(iii)} $\mu(G)-1\leq\mu(G-a)\leq\mu(G)$ for each $a\in V(G)\cup E(G)$.
\end{lemma}

Notice that: each edge of $C_{2q+1}$ is $\alpha$-critical and non-$\mu
$-critical; each vertex of $C_{2q+1}$ is neither $\alpha$-critical and nor
$\mu$-critical; each edge of $C_{2q}$ is neither $\alpha$-critical and nor
$\mu$-critical; each vertex of $C_{2q}$ is $\mu$-critical, but non-$\alpha$-critical.

Using Lemma \ref{lem1}, one can easily get the following.

\begin{corollary}
\label{cor25}The equality $\alpha(G-e)+\mu(G-e)=\alpha(G)+\mu(G)$ holds if and
only if either the edge $e\in E(G)$ is both $\alpha$-critical and $\mu
$-critical or the edge $e\in E(G)$ is both non-$\alpha$-critical and non-$\mu$-critical.
\end{corollary}

For instance, consider the graphs from Figure \ref{fig12345}:

\begin{itemize}
\item $\alpha(G_{1}-a)+\mu(G_{1}-a)=\alpha(G_{1})+\mu(G_{1})=6$, since $a$ is
both $\alpha$-critical and $\mu$-critical;

\item $\alpha(G_{2}-u_{1})+\mu(G_{2}-u_{1})=\alpha(G_{2})+\mu(G_{2})+1=7$,
since $u_{1}$ is $\alpha$-critical and non-$\mu$-critical;

\item $\alpha(G_{2}-u_{2})+\mu(G_{2}-u_{2})=\alpha(G_{2})+\mu(G_{2})=6$, since
$u_{2}$ is both non-$\alpha$-critical and non-$\mu$-critical;

\item $\alpha(G_{3}-b)+\mu(G_{3}-b)=\alpha(G_{3})+\mu(G_{3})-1=4$, since $b$
is non-$\alpha$-critical and $\mu$-critical.
\end{itemize}

\begin{figure}[h]
\setlength{\unitlength}{1cm}\begin{picture}(5,1.2)\thicklines
\multiput(1,0)(1,0){3}{\circle*{0.29}}
\multiput(1,1)(1,0){3}{\circle*{0.29}}
\put(1,0){\line(1,0){2}}
\put(1,0){\line(0,1){1}}
\put(1,1){\line(1,0){1}}
\put(2,1){\line(1,-1){1}}
\put(3,0){\line(0,1){1}}
\put(1.5,1.3){\makebox(0,0){$a$}}
\put(0.2,0.5){\makebox(0,0){$G_{1}$}}
\multiput(5,0)(1,0){4}{\circle*{0.29}}
\multiput(6,1)(1,0){3}{\circle*{0.29}}
\put(5,0){\line(1,0){3}}
\put(5,0){\line(1,1){1}}
\put(6,0){\line(0,1){1}}
\put(7,1){\line(1,-1){1}}
\put(7,1){\line(1,0){1}}
\put(8,0){\line(0,1){1}}
\put(5.15,0.6){\makebox(0,0){$u_{1}$}}
\put(7.2,0.5){\makebox(0,0){$u_{2}$}}
\put(4.2,0.5){\makebox(0,0){$G_{2}$}}
\multiput(10,0)(1,0){4}{\circle*{0.29}}
\multiput(11,1)(1,0){2}{\circle*{0.29}}
\put(10,0){\line(1,0){3}}
\put(10,0){\line(1,1){1}}
\put(11,0){\line(0,1){1}}
\put(12,1){\line(1,-1){1}}
\put(12,0){\line(0,1){1}}
\put(11.5,0.35){\makebox(0,0){$b$}}
\put(9.2,0.5){\makebox(0,0){$G_{3}$}}
\end{picture}\caption{Only $G_{1}$ is a K\"{o}nig-Egerv\'{a}ry graph.}%
\label{fig12345}%
\end{figure}
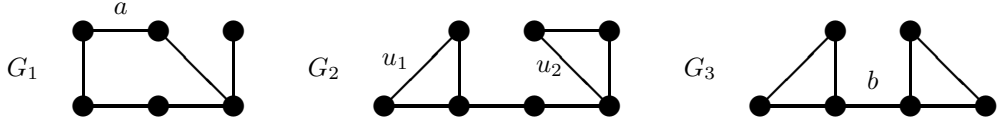

\begin{proposition}
\label{prop1_1}\cite{LevMan2006} In a K\"{o}nig-Egerv\'{a}ry graph $\alpha
$-critical edges are also $\mu$-critical, and they coincide in bipartite graphs.
\end{proposition}

Let $\varrho_{v}\left(  G\right)  $ denote the number of vertices $v\in
V\left(  G\right)  $, such that $G-v$ is a K\"{o}nig-Egerv\'{a}ry graph, and
$\varrho_{e}\left(  G\right)  $ denote the number of edges $e\in E\left(
G\right)  $ satisfying $G-e$ is a K\"{o}nig-Egerv\'{a}ry graph
\cite{LevMan2024}.

\begin{theorem}
\label{th9}\cite{LevMan2024} If $G$ is a K\"{o}nig-Egerv\'{a}ry graph, then%
\[
\varrho_{v}\left(  G\right)  =n\left(  G\right)  -\xi\left(  G\right)
+\varepsilon\left(  G\right)  \text{ and }\varrho_{e}\left(  G\right)  \leq
m\left(  G\right)  -\xi\left(  G\right)  +\varepsilon\left(  G\right)  .
\]

\end{theorem}

In this paper, we characterize $1$-K\"{o}nig-Egerv\'{a}ry graphs, vertex
(edge) almost K\"{o}nig-Egerv\'{a}ry graphs, and present interrelationships
between them. We also show that if $G$ is a $1$-K\"{o}nig-Egerv\'{a}ry graph,
then
\[
\varrho_{v}\left(  G\right)  \leq n\left(  G\right)  +d\left(  G\right)
-\xi\left(  G\right)  -\beta(G).
\]
As an application,{\LARGE ~}we characterize the $1$-K\"{o}nig-Egerv\'{a}ry
graphs that become K\"{o}nig-Egerv\'{a}ry after deleting any vertex. In other
words, for a $1$-K\"{o}nig-Egerv\'{a}ry graph $G$, the equality $\varrho
_{v}\left(  G\right)  =n\left(  G\right)  $ holds if and only if $\xi\left(
G\right)  =0$, $\beta\left(  G\right)  =0$ and $\mu\left(  G\right)
=\frac{n\left(  G\right)  }{2}$.

\section{Structural results}

\begin{definition}
A set $A\subseteq V(G)$ is \textit{supportive} if either

\emph{(i)} \textit{there is a vertex }$v\in V\left(  G\right)  -A$ and a
matching from $V\left(  G\right)  -A-v$ into $A$, or

\emph{(ii)}\textrm{ }there is an edge $xy$ $\in E\left(  G-A\right)  $ and a
matching from $V\left(  G\right)  -A-x-y$ into $A$.
\end{definition}

For instance, consider the graphs in Figure \ref{fig9}: the set $A=\left\{
a_{1},a_{2}\right\}  $ is a supportive (maximum independent) set in $G_{1}$,
and $B=\left\{  b_{1},b_{2}\right\}  $ is a supportive set in $G_{2}%
$.\begin{figure}[h]
\setlength{\unitlength}{1cm}\begin{picture}(5,1.2)\thicklines
\multiput(3,0)(1,0){3}{\circle*{0.29}}
\multiput(3,1)(2,0){2}{\circle*{0.29}}
\put(3,0){\line(1,0){2}}
\put(3,1){\line(1,0){2}}
\put(3,0){\line(0,1){1}}
\put(3,1){\line(1,-1){1}}
\put(4,0){\line(1,1){1}}
\put(5,0){\line(0,1){1}}
\put(4,0.35){\makebox(0,0){$v$}}
\put(2.65,0){\makebox(0,0){$a_{1}$}}
\put(5.35,0){\makebox(0,0){$a_{2}$}}
\put(2,0.5){\makebox(0,0){$G_{1}$}}
\multiput(8,0)(1,0){3}{\circle*{0.29}}
\multiput(8,1)(1,0){3}{\circle*{0.29}}
\put(8,0){\line(1,0){2}}
\put(8,1){\line(1,0){2}}
\put(8,0){\line(0,1){1}}
\put(8,1){\line(1,-1){1}}
\put(8,0){\line(1,1){1}}
\put(9,0){\line(1,1){1}}
\put(9,0){\line(0,1){1}}
\put(10,0){\line(0,1){1}}
\put(7.65,0){\makebox(0,0){$x$}}
\put(7.65,1){\makebox(0,0){$y$}}
\put(10.35,0){\makebox(0,0){$b_{1}$}}
\put(10.35,1){\makebox(0,0){$b_{2}$}}
\put(7,0.5){\makebox(0,0){$G_{2}$}}
\end{picture}\caption{Supportive sets: $\left\{  a_{1},a_{2}\right\}  $ and
$\left\{  b_{1},b_{2}\right\}  $.}%
\label{fig9}%
\end{figure}
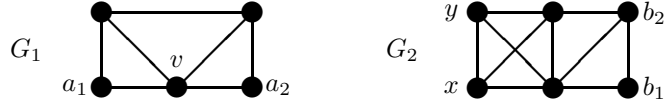

The following finding gives a structural characterization of $1$%
-K\"{o}nig-Egerv\'{a}ry graphs, similarly to some for K\"{o}nig-Egerv\'{a}ry
graphs \cite{LevMan2012b,LevMan2013b}.

\begin{theorem}
\label{Th1}Let $G$ be a non-K\"{o}nig-Egerv\'{a}ry graph. Then the following
assertions are equivalent:

\emph{(i)}\textrm{ }$G$ is a $1$-K\"{o}nig-Egerv\'{a}ry graph\textit{;}

\emph{(ii)} \textit{there exists a supportive maximum independent set in }%
$G$\textit{;}

\emph{(iii)}\textrm{ }\textit{every maximum independent set of }$G$ is
supportive\textit{.}
\end{theorem}

\begin{proof}
\emph{(i)}\textrm{ }$\Rightarrow$ \emph{(iii)}\textrm{ }Assume that $G$ is
$1$-K\"{o}nig-Egerv\'{a}ry, $S$ is an arbitrary maximum independent set, and
$M$ is a maximum matching\textit{. }Hence,\textit{ }%
\[
\mu\left(  G\right)  =n\left(  G\right)  -\alpha\left(  G\right)  -1=n\left(
G\right)  -\left\vert S\right\vert -1.
\]

Let $M$ contain $b_{1}$ edges connecting $S$ and $V\left(  G\right)  -S$,
while $b_{2}$ edges connecting vertices from $V\left(  G\right)  -S$. Thus
$\mu\left(  G\right)  =b_{1}+b_{2}$. Hence,
\[
n\left(  G\right)  -\alpha\left(  G\right)  =n\left(  G\right)  -\left\vert
S\right\vert \geq b_{1}+2b_{2}=\mu\left(  G\right)  +b_{2}.
\]
Therefore, we get
\[
1=n\left(  G\right)  -\alpha\left(  G\right)  -\mu\left(  G\right)  \geq
b_{2}.
\]

\textit{Case 1.} $b_{2}=0$. Since $\left\vert M\right\vert =$\ $\mu\left(
G\right)  =$\ \ $n\left(  G\right)  -\alpha\left(  G\right)  -1=$\ $\left\vert
V\left(  G\right)  -S\right\vert -1$, we infer that $M$ saturates all the
vertices from $V\left(  G\right)  -S$, except one, say $v\in V\left(
G\right)  -S$, and then $G-v$ is a K\"{o}nig-Egerv\'{a}ry graph.

\textit{Case 2.} $b_{2}=1$. Since $\left\vert M\right\vert -1=$\ $\mu\left(
G\right)  -1=$\ \ $n\left(  G\right)  -\alpha\left(  G\right)  -2=$%
\ $\left\vert V\left(  G\right)  -S\right\vert -2$, we conclude that $M$
saturates all the vertices from $V\left(  G\right)  -S$, and $M$ contains
exactly one edge that joins two vertices from $V\left(  G\right)  -S$.

Thus $S$ is supportive.

\emph{(iii)}\textrm{ }$\Rightarrow$ \emph{(ii)}\textrm{ }Clear.

\emph{(ii)}\textrm{ }$\Rightarrow$ \emph{(i)}\textrm{ }Let $S$ be a supportive
maximum independent set. Now, by the definition of a supportive set and the
fact that $S\in\Omega\left(  G\right)  $, we get $\alpha(G)+\mu(G)=n\left(
G\right)  -1$, as required.
\end{proof}

\begin{corollary}
\label{cor1}A non-K\"{o}nig-Egerv\'{a}ry\textit{ graph }$G$ is $1$%
-K\"{o}nig-Egerv\'{a}ry\textit{ if and only if either there is a vertex }$v\in
V\left(  G\right)  $ such that $G-v$ is a K\"{o}nig-Egerv\'{a}ry\textit{ graph
or there }is an edge $xy\in E\left(  G\right)  $ such that $G-x-y$ is a
K\"{o}nig-Egerv\'{a}ry graph\textit{.}
\end{corollary}

\begin{proof}
Suppose that $G$ is $1$-K\"{o}nig-Egerv\'{a}ry\textit{. }By Theorem \ref{Th1},
either there exists a vertex $v$ such that $G-v$ is K\"{o}nig-Egerv\'{a}ry or
there is an edge $xy$ such that $G-x-y$ is K\"{o}nig-Egerv\'{a}ry.

Conversely, assume that $G-v$ is a K\"{o}nig-Egerv\'{a}ry graph, for some
$v\in V\left(  G\right)  $. Since $G$ is not a K\"{o}nig-Egerv\'{a}ry graph,
we get that
\[
n\left(  G\right)  -1=\alpha(G-v)+\mu(G-v)\leq\alpha(G)+\mu(G)\leq n\left(
G\right)  -1\text{,}%
\]
which means that $\alpha(G)+\mu(G)=n\left(  G\right)  -1$, i.e., $G$ is $1$-K\"{o}nig-Egerv\'{a}ry.

Now, let $xy\in E(G)$ be such that $G-x-y$ is a K\"{o}nig-Egerv\'{a}ry graph.
Hence,
\[
n\left(  G\right)  -2=\alpha(G-x-y)+\mu(G-x-y)\leq\alpha(G)+\mu(G)\leq
n\left(  G\right)  -1\text{.}%
\]
It is clear that $\alpha(G)-\alpha(G-x-y)\geq0$. On the other hand,
$\mu(G)-\mu(G-x-y)>0$, because $xy\in E(G)$. Thus $\alpha(G)+\mu(G)=n\left(
G\right)  -1$, as required.
\end{proof}

For instance, the graphs from Figure \ref{fig111} are $1$%
-K\"{o}nig-Egerv\'{a}ry, as both $G_{1}-x-y$ and $G_{2}-v$ are
K\"{o}nig-Egerv\'{a}ry\textit{ }graphs\textit{.}

\begin{figure}[h]
\setlength{\unitlength}{1cm}\begin{picture}(5,2)\thicklines
\multiput(3,1)(1,0){4}{\circle*{0.29}}
\multiput(4,2)(1,0){2}{\circle*{0.29}}
\put(3,1){\line(1,0){3}}
\put(3,1){\line(1,1){1}}
\put(3,1){\line(2,1){2}}
\put(4,1){\line(0,1){1}}
\put(4,1){\line(1,1){1}}
\put(4,2){\line(1,-1){1}}
\put(4,2){\line(2,-1){2}}
\put(5,1){\line(0,1){1}}
\put(5,2){\line(1,-1){1}}
\qbezier(3,1)(4.5,-0.5)(6,1)
\qbezier(3,1)(4,0.3)(5,1)
\qbezier(4,1)(5,0.3)(6,1)
\put(3,1.35){\makebox(0,0){$x$}}
\put(6,1.35){\makebox(0,0){$y$}}
\put(2,1){\makebox(0,0){$G_{1}$}}
\multiput(9,1)(1,0){4}{\circle*{0.29}}
\multiput(9,2)(1,0){3}{\circle*{0.29}}
\put(9,1){\line(1,0){3}}
\put(9,1){\line(0,1){1}}
\put(9,1){\line(1,1){1}}
\put(9,1){\line(2,1){2}}
\put(9,2){\line(1,-1){1}}
\put(9,2){\line(2,-1){2}}
\put(9,2){\line(3,-1){3}}
\put(10,1){\line(0,1){1}}
\put(10,1){\line(1,1){1}}
\put(10,2){\line(1,-1){1}}
\put(10,2){\line(2,-1){2}}
\put(11,1){\line(0,1){1}}
\put(11,2){\line(1,-1){1}}
\qbezier(9,1)(10.5,-0.5)(12,1)
\qbezier(9,1)(10,0.3)(11,1)
\qbezier(10,1)(11,0.3)(12,1)
\put(12,1.35){\makebox(0,0){$v$}}
\put(8,1){\makebox(0,0){$G_{2}$}}
\end{picture}\caption{Both $G_{1}$ and $G_{2}$ are $1$-K\"{o}nig-Egerv\'{a}ry
graphs.}%
\label{fig111}%
\end{figure}
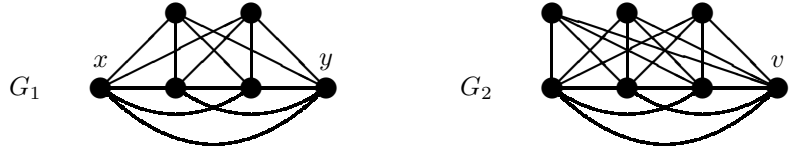

Clearly, Theorem \ref{Th1} may be reformulated as follows.

\begin{theorem}
\label{th911}Let $G$ be a non-K\"{o}nig-Egerv\'{a}ry graph. Then $G$ is a
$1$-K\"{o}nig-Egerv\'{a}ry graph\textit{ if and only if }\textrm{for
}\textit{every }$S\in\Omega\left(  G\right)  $\textit{, there is a vertex
}$v\in V\left(  G\right)  -S$ and a matching from $V\left(  G\right)
-S-\left\{  v\right\}  $ into $S$, or\textit{\ }there is an edge $xy$ $\in
E\left(  G-S\right)  $ and a matching from $V\left(  G\right)  -S-\left\{
x,y\right\}  $ into $S$.
\end{theorem}

It is known that $\mu(G)\leq\alpha(G)$ is true for every
K\"{o}nig-Egerv\'{a}ry graph (Theorem \ref{th715}\emph{(iii)})\emph{.
}Consider the graphs in Figure \ref{fig222} in order to see that for
$1$-K\"{o}nig-Egerv\'{a}ry graphs the situation is different.

\begin{figure}[h]
\setlength{\unitlength}{1cm}\begin{picture}(5,2.2)\thicklines
\multiput(2,0)(1,0){3}{\circle*{0.29}}
\multiput(2,1)(1,0){3}{\circle*{0.29}}
\multiput(3,2)(1,0){2}{\circle*{0.29}}
\put(2,0){\line(1,0){2}}
\put(2,0){\line(0,1){1}}
\put(2,0){\line(1,1){1}}
\put(2,0){\line(1,2){1}}
\put(2,1){\line(1,1){1}}
\put(2,1){\line(1,-1){1}}
\put(2,1){\line(1,0){1}}
\put(3,0){\line(1,1){1}}
\put(3,0){\line(1,2){1}}
\put(3,1){\line(1,1){1}}
\put(3,1){\line(1,-1){1}}
\put(3,1){\line(1,0){1}}
\put(3,2){\line(1,0){1}}
\put(3,2){\line(1,-1){1}}
\put(3,2){\line(1,-2){1}}
\put(4,0){\line(0,1){2}}
\put(1,1){\makebox(0,0){$G_{1}$}}
\multiput(6,0)(1,0){3}{\circle*{0.29}}
\multiput(7,1)(1,0){2}{\circle*{0.29}}
\put(6,0){\line(1,0){2}}
\put(6,0){\line(1,1){1}}
\put(7,0){\line(0,1){1}}
\put(8,0){\line(0,1){1}}
\put(5.2,0.5){\makebox(0,0){$G_{2}$}}
\multiput(10,0)(1,0){4}{\circle*{0.29}}
\multiput(11,1)(1,0){2}{\circle*{0.29}}
\put(10,0){\line(1,0){3}}
\put(10,0){\line(1,1){1}}
\put(11,0){\line(0,1){1}}
\put(12,0){\line(0,1){1}}
\put(9.2,0.5){\makebox(0,0){$G_{3}$}}
\end{picture}\caption{$\mu(G_{1})=\alpha(G_{1})+1$, $\mu(G_{2})=\alpha(G_{2}%
)$, and $\mu(G_{3})\,<\alpha(G_{3})$.}%
\label{fig222}%
\end{figure}
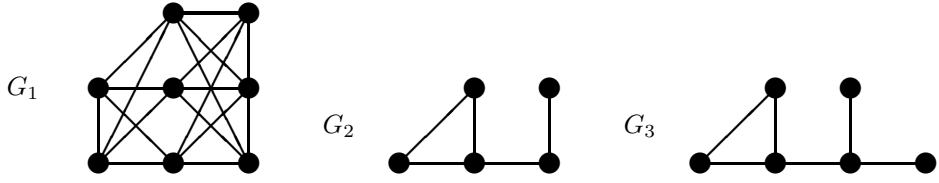

\begin{theorem}
\label{th12}If $G$ is a $1$-K\"{o}nig-Egerv\'{a}ry graph, then

\emph{(i)} $\mu(G)\leq\alpha(G)+1$;

\emph{(ii)} $\mu(G)=\alpha(G)+1$ if and only if $G$ has a perfect matching;

\emph{(iii)} $\mu(G)<\alpha(G)$, whenever $G$ has no perfect matchings and
$n\left(  G\right)  $ is even.
\end{theorem}

\begin{proof}
\emph{(i) }According to Theorem \ref{th911}, we distinguish between the
following cases.

\textit{Case 1}. There are a vertex\textit{ }$v\in V\left(  G\right)  -S$ and
a matching from $V\left(  G\right)  -S-\left\{  v\right\}  $ into $S$, where
$S\in\Omega\left(  G\right)  $.\textit{ }

By Theorem \ref{th715}, it follows that $\mu\left(  G-v\right)  =\mu\left(
G\right)  $ and $\alpha\left(  G-v\right)  =\alpha\left(  G\right)  $. Since
$G-v$ is a K\"{o}nig-Egerv\'{a}ry graph, we know that $\mu\left(  G-v\right)
\leq\alpha\left(  G-v\right)  $. Thus, $\mu(G)\leq\alpha(G)\leq\alpha(G)+1$.

\textit{Case 2}. There is an edge $xy$ $\in E\left(  G-S\right)  $ and a
matching from $V\left(  G\right)  -S-\left\{  x,y\right\}  $ into $S$, where
$S\in\Omega\left(  G\right)  $.\textit{ }

By Theorem \ref{th715}, it follows that $\mu\left(  G-\left\{  x,y\right\}
\right)  +1=\mu\left(  G\right)  $ and $\alpha\left(  G-\left\{  x,y\right\}
\right)  =\alpha\left(  G\right)  $. Since $G-\left\{  x,y\right\}  $ is a
K\"{o}nig-Egerv\'{a}ry graph, we obtain
\[
\mu\left(  G-\left\{  x,y\right\}  \right)  \leq\alpha\left(  G-\left\{
x,y\right\}  \right)  \Leftrightarrow\mu\left(  G\right)  -1\leq\alpha\left(
G\right)  ,
\]
as required.

\emph{(ii)} If $\mu(G)=\alpha(G)+1$, then $n\left(  G\right)  =\mu
(G)+\alpha(G)+1=2\mu(G)$, which means that $G$ has a perfect matching.

Conversely, if $G$ has a perfect matching, then $2\mu(G)=n\left(  G\right)
=\mu(G)+\alpha(G)+1$, and this gives $\mu(G)=\alpha(G)+1$.

\emph{(iii)} If $G$ has no perfect matchings, then $\mu(G)<\frac{n\left(
G\right)  }{2}$. Hence,%
\[
\frac{n\left(  G\right)  }{2}+\alpha(G)>\mu(G)+\alpha(G)=n\left(  G\right)
-1,
\]
which means that $\alpha(G)>\frac{n\left(  G\right)  }{2}-1$. Since
$\frac{n\left(  G\right)  }{2}$ is integer, we obtain $\alpha(G)\geq
\frac{n\left(  G\right)  }{2}$. Finally, we get $\alpha(G)\geq\frac{n\left(
G\right)  }{2}>\mu(G)$, which completes the proof.
\end{proof}

\begin{theorem}
\label{th3}If $G$ is either an edge almost K\"{o}nig-Egerv\'{a}ry graph or a
vertex almost K\"{o}nig-Egerv\'{a}ry graph, then $G$ is a $1$%
-K\"{o}nig-Egerv\'{a}ry graph as well.
\end{theorem}

\begin{proof}
Let $e\in E(G)$ be such that $G-e$ is a K\"{o}nig-Egerv\'{a}ry graph. Since
$G$ is not a K\"{o}nig-Egerv\'{a}ry graph, and clearly, $\alpha(G)\leq
\alpha(G-e)\leq\alpha(G)+1$ and $\mu(G)\geq\mu(G-e)$, we obtain%
\[
n\left(  G\right)  -1=n\left(  G-e\right)  -1=\alpha(G-e)-1+\mu(G-e)\leq
\alpha(G)+\mu(G)<n\left(  G\right)  \text{.}%
\]
Therefore, $\alpha(G)+\mu(G)=n\left(  G\right)  -1$, i.e., $G$ is a
$1$-K\"{o}nig-Egerv\'{a}ry graph.

Let $v\in V\left(  G\right)  $ be such that $G-v$ is a K\"{o}nig-Egerv\'{a}ry
graph. Since $G$ is not a K\"{o}nig-Egerv\'{a}ry graph, we deduce that
\[
n\left(  G\right)  -1=n\left(  G-v\right)  =\alpha(G-v)+\mu(G-v)\leq
\alpha(G)+\mu(G)<n\left(  G\right)  \text{,}%
\]
which implies that $\alpha(G)+\mu(G)=n\left(  G\right)  -1$, i.e., $G$ is a
$1$-K\"{o}nig-Egerv\'{a}ry graph.
\end{proof}

Notice that there exist $1$-K\"{o}nig-Egerv\'{a}ry graphs that are neither
edge almost K\"{o}nig-Egerv\'{a}ry graphs, nor vertex almost
K\"{o}nig-Egerv\'{a}ry graphs; e.g.\textit{,} $pK_{1}+K_{p+1}$, $pK_{1}%
+K_{p+2}$.

In continuation of Lemma \ref{Larson lemma}, we proceed with the following.

\begin{theorem}
\label{th2}\emph{(i)} A graph $G$ is vertex almost K\"{o}nig-Egerv\'{a}ry if
and only if it is $1$-K\"{o}nig-Egerv\'{a}ry and some $v\in V(G)$ is neither
$\alpha$-critical nor $\mu$-critical.

\emph{(ii)} A graph $G$ is edge almost K\"{o}nig-Egerv\'{a}ry if and only if
it is $1$-K\"{o}nig-Egerv\'{a}ry and some $e\in E(G)$ is $\alpha$-critical and
non-$\mu$-critical.
\end{theorem}

\begin{proof}
Since $G$ is not a K\"{o}nig-Egerv\'{a}ry graph, we know that $\alpha
(G)+\mu(G)<n\left(  G\right)  $.

\emph{(i) }Assume that\emph{ }$G$ is vertex almost K\"{o}nig-Egerv\'{a}ry. By
Theorem \ref{th3}, $G$ is also $1$-K\"{o}nig-Egerv\'{a}ry.

There exists $v\in V\left(  G\right)  $ such that $G-v$ is a
K\"{o}nig-Egerv\'{a}ry graph.\ \ According to Lemma \ref{lem1}\textit{(ii)}
and \emph{(iii)}, we get%
\[
n\left(  G\right)  -1=\alpha(G-v)+\mu(G-v)\leq\alpha(G)+\mu(G)<n\left(
G\right)  ,
\]
which implies that $\alpha(G-v)=\alpha(G)$ and $\mu(G-v)=\mu(G)$, and these
mean that $v\in V(G)$ is neither $\alpha$-critical nor $\mu$-critical.

The converse is clear, because $\alpha(G-v)+\mu(G-v)=\alpha(G)+\mu(G)=n\left(
G\right)  -1$.

\emph{(ii) }Suppose that\emph{ }$G$ is vertex almost K\"{o}nig-Egerv\'{a}ry.
Then, by Theorem \ref{th3}, $G$ is also $1$-K\"{o}nig-Egerv\'{a}ry.

There exists $xy\in E\left(  G\right)  $ such that $G-xy$ is a
K\"{o}nig-Egerv\'{a}ry graph.\ \ According to Lemma \ref{lem1}\textit{(i)} and
\emph{(iii)}, we get%
\[
n\left(  G\right)  =\alpha(G-xy)+\mu(G-xy)\leq\alpha(G)+1+\mu(G)<n\left(
G\right)  +1,
\]
which implies that $\alpha(G-xy)=\alpha\left(  G\right)  +1$ and
$\mu(G-xy)=\mu\left(  G\right)  $, i.e., the \ edge $xy$ is $\alpha$-critical
and non-$\mu$-critical.

Conversely, we have that $\alpha(G-xy)=\alpha\left(  G\right)  +1$,
$\mu(G-xy)=\mu\left(  G\right)  $, and $\alpha(G)+\mu(G)=n\left(  G\right)
-1$, which ensures that
\[
\alpha(G-xy)+\mu(G-xy)=\alpha(G)+\mu(G)+1=n\left(  G\right)  ,
\]
and this means that $G$ is an edge almost K\"{o}nig-Egerv\'{a}ry graph.
\end{proof}

Recall that a graph is\textit{ almost bipartite} if it has a unique odd cycle
\cite{LevMan2022}.

\begin{lemma}
\label{lem84}\cite{LevMan2022} If $G$ is an almost bipartite graph, then
$n(G)-1\leq\alpha(G)+\mu(G)\leq n(G)$.
\end{lemma}

Consequently, one may say that each almost bipartite graph is either a
K\"{o}nig-Egerv\'{a}ry graph or a $1$-K\"{o}nig-Egerv\'{a}ry graph.

\begin{corollary}
\label{cor3}If $G$\ is an almost bipartite graph, then the following
assertions are equivalent:

\emph{(i)} $G$ is a $1$-K\"{o}nig-Egerv\'{a}ry graph;

\emph{(ii)} $G$ is a vertex \textit{almost} K\"{o}nig-Egerv\'{a}ry\textit{
graph};

\emph{(iii)} $G$ is an edge almost K\"{o}nig-Egerv\'{a}ry\textit{ graph}.
\end{corollary}

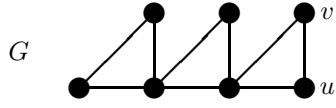
\begin{figure}[h]
\setlength{\unitlength}{1cm}\begin{picture}(5,1.2)\thicklines
\multiput(5,0)(1,0){4}{\circle*{0.29}}
\multiput(6,1)(1,0){3}{\circle*{0.29}}
\put(5,0){\line(1,0){3}}
\put(5,0){\line(1,1){1}}
\put(6,0){\line(1,1){1}}
\put(6,0){\line(0,1){1}}
\put(7,0){\line(1,1){1}}
\put(7,0){\line(0,1){1}}
\put(8,0){\line(0,1){1}}
\put(8.3,0){\makebox(0,0){$u$}}
\put(8.3,1){\makebox(0,0){$v$}}
\put(4.2,0.5){\makebox(0,0){$G$}}
\end{picture}\caption{A non almost bipartite $1$-K\"{o}nig-Egerv\'{a}ry
graph.}%
\label{fig22}%
\end{figure}

By Theorem \ref{th2}, $G$ from Figure \ref{fig22} is both a vertex and an edge
almost K\"{o}nig-Egerv\'{a}ry\textit{ }graph (since the vertex $v$ is neither
$\alpha$-critical nor $\mu$-critical, while the edge $uv$ is $\alpha$-critical
and non-$\mu$-critical).

\section{The $\varrho_{v}$-inequality for $1$-K\"{o}nig-Egerv\'{a}ry graphs}

Notice that for every $n\geq3$, the complete graph $K_{n}$ is not
K\"{o}nig-Egerv\'{a}ry; the same is true for $K_{n}-v$, whenever $n\geq4$.
However, $K_{3}-v$ is a K\"{o}nig-Egerv\'{a}ry graph, while $K_{4}-v$ is a
$1$-K\"{o}nig-Egerv\'{a}ry graph, for every vertex $v$.

\begin{theorem}
\label{th8}If there is a vertex $v\in V\left(  G\right)  $, such that $G-v$ is
a K\"{o}nig-Egerv\'{a}ry graph, then $G$ is either K\"{o}nig-Egerv\'{a}ry or
$1$-K\"{o}nig-Egerv\'{a}ry.
\end{theorem}

\begin{proof}
By definition, we know that
\[
\alpha(G-v)+\mu(G-v)=n\left(  G-v\right)  =n\left(  G\right)  -1.
\]

\textit{Case 1.} $\alpha(G-v)=\alpha(G)$.

Thus $\alpha(G)+\mu(G-v)=n\left(  G\right)  -1$. If $\mu(G-v)=\mu(G)$, then
$\alpha(G)+\mu(G)=n\left(  G\right)  -1$, which means that $G$ is a
$1$-K\"{o}nig-Egerv\'{a}ry graph; otherwise, $\mu(G-v)=\mu(G)-1$ and,
consequently, $\alpha(G)+\mu(G)=n\left(  G\right)  $, i.e., $G$ is a
K\"{o}nig-Egerv\'{a}ry graph.

\textit{Case 2.} $\alpha(G-v)=\alpha(G)-1$.

Thus $\alpha(G)-1+\mu(G-v)=n\left(  G\right)  -1$. If $\mu(G-v)=\mu(G)$, then
$\alpha(G)+\mu(G)=n\left(  G\right)  $, which means that $G$ is a
K\"{o}nig-Egerv\'{a}ry graph; otherwise, $\mu(G-v)=\mu(G)-1$ and,
consequently, $\alpha(G)+\mu(G)=n\left(  G\right)  +1$, which is impossible,
as $\alpha(G)+\mu(G)\leq n\left(  G\right)  $ for every graph $G$.
\end{proof}

\begin{corollary}
If $G$ is a K\"{o}nig-Egerv\'{a}ry graph, then $G+v$, where $v\notin V\left(
G\right)  $ and $N_{G+v}\left(  v\right)  =A\subseteq V\left(  G\right)  $, is
either K\"{o}nig-Egerv\'{a}ry or $1$-K\"{o}nig-Egerv\'{a}ry.
\end{corollary}

Consider the $1$-K\"{o}nig-Egerv\'{a}ry graphs from Figure \ref{fig125}.
$G_{1}-a$ is a K\"{o}nig-Egerv\'{a}ry graph and the vertex $a$ is neither
$\alpha$-critical nor\textit{ }$\mu$-critical, while $G_{1}-b$ is a not a
K\"{o}nig-Egerv\'{a}ry graph and the vertex $b$ is $\alpha$-critical. The
vertex $x$ is both $\alpha$-critical and\textit{ }$\mu$-critical, and
$G_{2}-x$ is a not a K\"{o}nig-Egerv\'{a}ry graph.

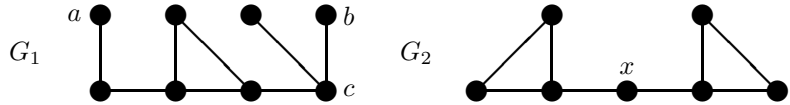
\begin{figure}[h]
\setlength{\unitlength}{1cm}\begin{picture}(5,1.2)\thicklines
\multiput(3,0)(1,0){4}{\circle*{0.29}}
\multiput(3,1)(1,0){4}{\circle*{0.29}}
\put(3,0){\line(1,0){3}}
\put(3,0){\line(0,1){1}}
\put(4,0){\line(0,1){1}}
\put(4,1){\line(1,-1){1}}
\put(5,1){\line(1,-1){1}}
\put(6,0){\line(0,1){1}}
\put(2.65,1){\makebox(0,0){$a$}}
\put(6.3,1){\makebox(0,0){$b$}}
\put(6.3,0){\makebox(0,0){$c$}}
\put(2,0.5){\makebox(0,0){$G_{1}$}}
\multiput(8,0)(1,0){5}{\circle*{0.29}}
\multiput(9,1)(2,0){2}{\circle*{0.29}}
\put(8,0){\line(1,0){4}}
\put(8,0){\line(1,1){1}}
\put(9,0){\line(0,1){1}}
\put(11,0){\line(0,1){1}}
\put(11,1){\line(1,-1){1}}
\put(10,0.3){\makebox(0,0){$x$}}
\put(7.2,0.5){\makebox(0,0){$G_{2}$}}
\end{picture}\caption{$\alpha$-critical vertices and $\mu$-critical vertices
in $1$-K\"{o}nig-Egerv\'{a}ry graphs.}%
\label{fig125}%
\end{figure}

\begin{theorem}
\label{th17}Let $G$ be a $1$-K\"{o}nig-Egerv\'{a}ry graph. Then $G-v$ is
K\"{o}nig-Egerv\'{a}ry if and only if the vertex $v$ is neither $\alpha
$-\textit{critical nor }$\mu$-\textit{critical.}
\end{theorem}

\begin{proof}
According to the definition, we know that $\alpha(G)+\mu(G)=n\left(  G\right)
-1$.

Assume that $G-v$ is a K\"{o}nig-Egerv\'{a}ry graph. Then we obtain%
\[
\alpha\left(  G-v\right)  +\mu\left(  G-v\right)  =n\left(  G-v\right)
=n\left(  G\right)  -1=\alpha(G)+\mu(G).
\]
Since both $\alpha\left(  G-v\right)  \leq\alpha\left(  G\right)  $ and
$\mu\left(  G-v\right)  \leq\mu\left(  G\right)  $, we get $\alpha\left(
G-v\right)  =\alpha\left(  G\right)  $ and $\mu\left(  G-v\right)  =\mu\left(
G\right)  $, i.e., the vertex $v$ is neither $\alpha$-critical nor $\mu$-critical.

Conversely, if $v$ is neither $\alpha$-critical nor $\mu$-critical, then
\[
\alpha\left(  G-v\right)  +\mu\left(  G-v\right)  =\alpha(G)+\mu(G)=n\left(
G\right)  -1=n\left(  G-v\right)  ,
\]
which means that $G-v$ is a K\"{o}nig-Egerv\'{a}ry graph.
\end{proof}

\begin{theorem}
\label{th11} Suppose that $A\in\mathrm{Ind}(G)$. If there is a matching from
$N_{G}(A)$ into $A$, then\textit{ every matching }from $N_{G}(A)$ into $A$ can
be enlarged to a maximum matching of $G$, and every vertex of $N_{G}(A)$ is
$\mu$-\textit{critical.}
\end{theorem}

\begin{proof}
Let $M$ be a maximum matching of $G$, $M\left(  N_{G}(A)\right)  $ be the
vertices of $N_{G}(A)$ that are saturated by $M$, and $M_{1}$ be a
matching\textit{ }from $N_{G}(A)$ into $A$. If $\left\vert M\left(
N_{G}(A)\right)  \right\vert <\left\vert N_{G}(A)\right\vert $, then $M\cup
M_{1}-M_{2}$, where $M_{2}=\left\{  xy:xy\in M\text{ and }x\in M\left(
N_{G}(A)\right)  \right\}  $, is a matching of $G$ larger than $M$,
contradicting the maximality of $M$. Therefore, $\left\vert M\left(
N_{G}(A)\right)  \right\vert =\left\vert N_{G}(A)\right\vert $, and $M\cup
M_{1}-M_{2}$ is a maximum matching of $G$ that enlarges $M_{2}$.

In other words, we have
\[
\mu\left(  G\right)  =\mu\left(  G\left[  N_{G}\left[  A\right]  \right]
\right)  +\mu\left(  G\left[  V\left(  G\right)  -N_{G}\left[  A\right]
\right]  \right)  ,
\]
which means that every matching from $N_{G}(A)$ into $A$ can be enlarged to a
maximum matching of $G$.

Second, let $v\in N_{G}(A)$ and $H=G-v$. Then there is a matching from
$N_{H}(A)=$ $N_{G}(A)-\left\{  v\right\}  $ into $A$, and consequently,
\begin{gather*}
\mu\left(  G-v\right)  =\mu\left(  H\right)  =\mu\left(  H\left[  N_{H}\left[
A\right]  \right]  \right)  +\mu\left(  H\left[  V\left(  H\right)
-N_{H}\left[  A\right]  \right]  \right)  =\\
\mu\left(  G\left[  N_{G}\left[  A\right]  -\left\{  v\right\}  \right]
\right)  +\mu\left(  G\left[  V\left(  G\right)  -N_{G}\left[  A\right]
\right]  \right)  =\\
\mu\left(  G\left[  N_{G}\left[  A\right]  \right]  \right)  -1+\mu\left(
G\left[  V\left(  G\right)  -N_{G}\left[  A\right]  \right]  \right)
=\mu\left(  G\right)  -1.
\end{gather*}
Thus, $v$ is a $\mu$-critical vertex.
\end{proof}

Let us consider the graphs from Figure \ref{fig55}. The set $\left\{
a_{1},a_{4}\right\}  \in\mathrm{Ind}(G_{1})$ and there is a matching from
$N\left(  \left\{  a_{1},a_{4}\right\}  \right)  $ \ into $\left\{
a_{1},a_{4}\right\}  $; consequently, this matching is included in a maximum
matching of $G_{1}$. On the other hand, the set $\left\{  v\right\}
\in\mathrm{Ind}(G_{2})$ and there no a matching from $N\left(  \left\{
v\right\}  \right)  $ \ into $\left\{  v\right\}  $; however, every matching
of $G_{2}$ can be enlarged to a maximum matching. This fact shows that the
converse of Theorem \ref{th11} is not generally true.

\begin{figure}[h]
\setlength{\unitlength}{1cm}\begin{picture}(5,1.2)\thicklines
\multiput(2,0)(1,0){4}{\circle*{0.29}}
\multiput(3,1)(1,0){2}{\circle*{0.29}}
\put(2,0){\line(1,0){3}}
\put(3,1){\line(1,0){1}}
\put(3,0){\line(0,1){1}}
\put(4,1){\line(1,-1){1}}
\put(2,0.35){\makebox(0,0){$a_{1}$}}
\put(2.65,1){\makebox(0,0){$a_{2}$}}
\put(2.65,0.35){\makebox(0,0){$a_{4}$}}
\put(4.35,1){\makebox(0,0){$a_{3}$}}
\put(4,0.35){\makebox(0,0){$a_{5}$}}
\put(5.35,0){\makebox(0,0){$a_{6}$}}
\put(1,0.5){\makebox(0,0){$G_{1}$}}
\multiput(8,0)(1,0){4}{\circle*{0.29}}
\multiput(9,1)(1,0){2}{\circle*{0.29}}
\put(8,0){\line(1,0){3}}
\put(9,0){\line(0,1){1}}
\put(10,1){\line(1,-1){1}}
\put(10,0){\line(0,1){1}}
\put(9.65,1){\makebox(0,0){$v$}}
\put(7,0.5){\makebox(0,0){$G_{2}$}}
\end{picture}\caption{$\mu\left(  G_{1}\right)  =3$ and $\mu\left(
G_{2}\right)  =2$.}%
\label{fig55}%
\end{figure}
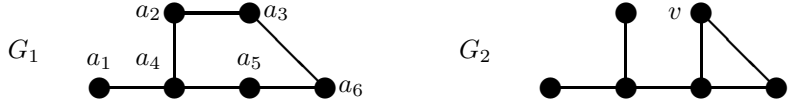

Theorem \ref{th11} and Theorem \ref{th3}\emph{(iii)} immediately imply the following.

\begin{corollary}
If $A$ is a critical independent set in a graph $G$, then\textit{ every
matching }from $N_{G}(A)$ into $A$ can be enlarged to a maximum matching of
$G$, and every vertex of $N_{G}(A)$ is $\mu$-\textit{critical.}
\end{corollary}

\begin{proposition}
\label{prop11}If $A$ is a critical independent set in $G$, then $\emph{core}%
\left(  G\right)  \cap N_{G}\left(  A\right)  =\emptyset$, and, consequently,
$\emph{core}\left(  G\right)  \cap N_{G}\left(  \emph{diadem}\left(  G\right)
\right)  =\emptyset$.
\end{proposition}

\begin{proof}
By Theorem \ref{th3}, there exists a maximum independent set $S$, such that
$A\subseteq S$. By definition, $\emph{core}\left(  G\right)  \subseteq S$, as
well. Hence,
\[
\emph{core}\left(  G\right)  \cap N_{G}\left(  A\right)  \subseteq
\emph{core}\left(  G\right)  \cap N_{G}\left(  S\right)  =\emptyset.
\]
Thus, $\emph{core}\left(  G\right)  \cap N_{G}\left(  A\right)  =\emptyset$.

Finally, we obtain
\begin{align*}
\emph{core}\left(  G\right)  \cap N_{G}\left(  \emph{diadem}\left(  G\right)
\right)   &  =\emph{core}\left(  G\right)  \cap N\left(
{\displaystyle\bigcup\limits_{A\in Crit(G)}}
A\right)  =\\
\emph{core}\left(  G\right)  \cap%
{\displaystyle\bigcup\limits_{A\in Crit(G)}}
N\left(  A\right)   &  =%
{\displaystyle\bigcup\limits_{A\in Crit(G)}}
\left(  \emph{core}\left(  G\right)  \cap N\left(  A\right)  \right)
=\emptyset,
\end{align*}
as claimed.
\end{proof}

\begin{example}
\label{example1}Consider the $1$-K\"{o}nig-Egerv\'{a}ry graphs from Figure
\ref{fig124}: $n\left(  G_{1}\right)  =n\left(  G_{2}\right)  =8$, $d\left(
G_{1}\right)  =d\left(  G_{2}\right)  =1$, $\xi\left(  G_{1}\right)
=\xi\left(  G_{2}\right)  =2$, $\alpha^{\prime}\left(  G_{1}\right)
=\alpha^{\prime}\left(  G_{2}\right)  =3$, while $\varrho\left(  G_{1}\right)
=4$ and $\varrho\left(  G_{2}\right)  =3$.
\end{example}

\begin{figure}[h]
\setlength{\unitlength}{1cm}\begin{picture}(5,1.2)\thicklines
\multiput(3,0)(1,0){4}{\circle*{0.29}}
\multiput(3,1)(1,0){4}{\circle*{0.29}}
\put(3,0){\line(1,0){3}}
\put(3,0){\line(0,1){1}}
\put(3,1){\line(1,-1){1}}
\put(4,1){\line(1,-1){1}}
\put(5,1){\line(1,-1){1}}
\put(6,0){\line(0,1){1}}
\put(3.65,1){\makebox(0,0){$a$}}
\put(4.65,1){\makebox(0,0){$b$}}
\put(6.3,1){\makebox(0,0){$c$}}
\put(6.3,0){\makebox(0,0){$d$}}
\put(2.2,0.5){\makebox(0,0){$G_{1}$}}
\multiput(8,0)(1,0){4}{\circle*{0.29}}
\multiput(8,1)(1,0){4}{\circle*{0.29}}
\put(8,0){\line(1,0){3}}
\put(8,0){\line(0,1){1}}
\put(8,1){\line(1,-1){1}}
\put(9,1){\line(1,-1){1}}
\put(10,0){\line(0,1){1}}
\put(11,0){\line(0,1){1}}
\put(8.65,1){\makebox(0,0){$u$}}
\put(10.3,1){\makebox(0,0){$v$}}
\put(10.3,0.3){\makebox(0,0){$w$}}
\put(11.3,1){\makebox(0,0){$x$}}
\put(11.3,0){\makebox(0,0){$y$}}
\put(7.2,0.5){\makebox(0,0){$G_{2}$}}
\end{picture}\caption{$1$-K\"{o}nig-Egerv\'{a}ry graphs with $\varrho
(G_{1})=4$ and $\varrho(G_{2})=3$.}%
\label{fig124}%
\end{figure}
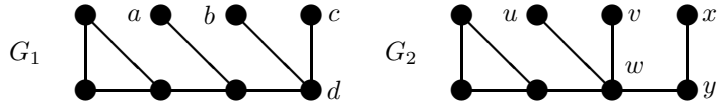

\begin{theorem}
\label{cor18}If $G$ is a $1$-K\"{o}nig-Egerv\'{a}ry graph, then
\[
\varrho_{v}\left(  G\right)  \leq n\left(  G\right)  +d\left(  G\right)
-\xi\left(  G\right)  -\beta(G).
\]

\end{theorem}

\begin{proof}
Clearly, $\emph{diadem}\left(  G\right)  =%
{\displaystyle\bigcup\limits_{A\in Crit(G)}}
A=%
{\displaystyle\bigcup\limits_{A\in MaxCrit(G)}}
A$, since every critical set is a subset of a maximum critical set (Theorem
\ref{th3}\emph{(ii)}). Hence, we infer that%
\[
N_{G}\left(  \emph{diadem}\left(  G\right)  \right)  =N\left(
{\displaystyle\bigcup\limits_{A\in MaxCrit(G)}}
A\right)  =%
{\displaystyle\bigcup\limits_{A\in MaxCrit(G)}}
N_{G}\left(  A\right)  .
\]
Then, Theorem \ref{th17}, Theorem \ref{th11} and Proposition \ref{prop11}
imply
\[
\varrho_{v}\left(  G\right)  \leq n\left(  G\right)  -\xi\left(  G\right)
-\left\vert
{\displaystyle\bigcup\limits_{A\in MaxCrit(G)}}
N_{G}\left(  A\right)  \right\vert =n\left(  G\right)  -\xi\left(  G\right)
-\left\vert N_{G}\left(  \emph{diadem}\left(  G\right)  \right)  \right\vert
.
\]
By Theorem \ref{th444}, $\emph{diadem}\left(  G\right)  $ is critical in $G$,
as union of critical sets. Hence, $d\left(  G\right)  =\left\vert
\emph{diadem}\left(  G\right)  \right\vert -\left\vert N_{G}\left(
\emph{diadem}\left(  G\right)  \right)  \right\vert $. Thus,%
\[
\varrho_{v}\left(  G\right)  \leq n\left(  G\right)  +d\left(  G\right)
-\xi\left(  G\right)  -\beta(G),
\]
and this completes the proof.
\end{proof}

It is worth mentioning that the $1$-K\"{o}nig-Egerv\'{a}ry graphs in Figure
\ref{fig101} point out to the fact that Theorem \ref{cor18} presents a tight inequality:

\begin{itemize}
\item $n\left(  G_{1}\right)  =9$, $\alpha\left(  G_{1}\right)  =\mu\left(
G_{1}\right)  =4$, $d\left(  G_{1}\right)  =\xi\left(  G_{1}\right)  =0$,
$\emph{diadem}\left(  G_{1}\right)  =\left\{  u,v,w\right\}  $, $\varrho
_{v}\left(  G_{1}\right)  =6=n\left(  G_{1}\right)  +d\left(  G_{1}\right)
-\xi\left(  G_{1}\right)  -\beta\left(  G_{1}\right)  $;

\item $n\left(  G_{2}\right)  =9$, $\alpha\left(  G_{2}\right)  =\mu\left(
G_{2}\right)  =4$, $d\left(  G_{2}\right)  =0$, $\xi\left(  G_{2}\right)  =1$,
$\emph{diadem}\left(  G_{2}\right)  =\left\{  a,b\right\}  $, $\varrho
_{v}\left(  G_{2}\right)  =5<n\left(  G_{2}\right)  +d\left(  G_{2}\right)
-\xi\left(  G_{2}\right)  -\beta\left(  G_{2}\right)  $.
\end{itemize}

\begin{figure}[h]
\setlength{\unitlength}{1cm}\begin{picture}(5,1.2)\thicklines
\multiput(2,0)(1,0){5}{\circle*{0.29}}
\multiput(3,1)(1,0){4}{\circle*{0.29}}
\put(2,0){\line(1,0){1}}
\put(2,0){\line(1,1){1}}
\put(3,0){\line(0,1){1}}
\put(3,1){\line(2,-1){2}}
\put(4,1){\line(1,0){2}}
\put(4,0){\line(0,1){1}}
\put(4,0){\line(1,0){2}}
\put(5,0){\line(0,1){1}}
\put(6,0){\line(0,1){1}}
\put(3.65,0){\makebox(0,0){$u$}}
\put(5.3,0.8){\makebox(0,0){$v$}}
\put(6.35,0){\makebox(0,0){$w$}}
\put(1,0.5){\makebox(0,0){$G_{1}$}}
\multiput(8,0)(1,0){5}{\circle*{0.29}}
\multiput(11,1)(1,0){2}{\circle*{0.29}}
\put(8,1){\circle*{0.29}}
\put(9,1){\circle*{0.29}}
\put(8,0){\line(1,0){4}}
\put(8,0){\line(0,1){1}}
\put(8,1){\line(1,-1){1}}
\put(9,0){\line(0,1){1}}
\put(11,1){\line(1,0){1}}
\put(10,0){\line(1,1){1}}
\put(12,0){\line(0,1){1}}
\put(8.35,1){\makebox(0,0){$a$}}
\put(9.35,1){\makebox(0,0){$b$}}
\put(7,0.5){\makebox(0,0){$G_{2}$}}
\end{picture}\caption{Tight inequality examples.}%
\label{fig101}%
\end{figure}
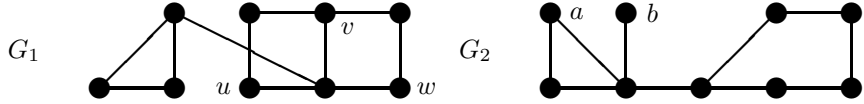

Theorem \ref{cor18} and the definition of $\alpha^{\prime}\left(  G\right)  $
imply the following.

\begin{corollary}
\label{cor2}If $G$ is $1$-K\"{o}nig-Egerv\'{a}ry graph, then
\[
\varrho\left(  G\right)  \leq n\left(  G\right)  +d\left(  G\right)
-\xi\left(  G\right)  -\alpha^{\prime}\left(  G\right)  .
\]

\end{corollary}

\begin{remark}
Example \ref{example1} shows that the inequality presented in Theorem
\ref{cor18} is stronger than its counterpart presented in Corollary
\ref{cor2}. To see it, pay attention to the fact that for the graph $G_{2}$ in
Figure \ref{fig124} the set $\left\{  u,v,y\right\}  $ is maximum critical
independent as well as $\left\{  u,v,x\right\}  $, which is different in
comparison with the graph $G_{1}$ in the same figure, where there is only one
maximum critical independent set $\left\{  a,b,c\right\}  $.
\end{remark}

If $G$ is a $1$-K\"{o}nig-Egerv\'{a}ry graph (a general graph - $K_{2q}$),
then $N\left(  \emph{diadem}\left(  G\right)  \right)  $ may be a proper
subset of the set of all $\mu$-critical vertices of $G$ that are not $\alpha
$-critical\textit{. }In other words, it is not true that a vertex\textit{ }%
$v$\textit{ }is\textit{ }$\mu$-critical and not\textit{ }$\alpha$-critical if
and only if there exists a (maximum) critical independent set $A$\textit{
}such that\textit{ }$v\in N_{G}\left(  A\right)  $.

\section{An application}

A graph $G$ is a \textit{critical vertex almost K\"{o}nig-Egerv\'{a}ry graph},
if $G$ is not a K\"{o}nig-Egerv\'{a}ry graph, but $G-v$ is a
K\"{o}nig-Egerv\'{a}ry graph for every $v\in V(G)$, i.e., $\varrho_{v}\left(
G\right)  =n\left(  G\right)  $. For instance, every $C_{2k+1}$ is a critical
vertex almost K\"{o}nig-Egerv\'{a}ry graph. It is worth reminding that, by
Theorem \ref{th3}, $G$ is $1$-K\"{o}nig-Egerv\'{a}ry.

\begin{theorem}
\label{lem10}The graph $G$ is \textit{critical vertex almost
K\"{o}nig-Egerv\'{a}ry if and only if }$G$ is $1$-K\"{o}nig-Egerv\'{a}ry
without perfect matchings, $\xi\left(  G\right)  =0$ and $\beta\left(
G\right)  =0$.
\end{theorem}

\begin{proof}
\emph{"If" Part.} Suppose that $G$ is critical vertex almost
K\"{o}nig-Egerv\'{a}ry.\textit{ } By Theorem \ref{th3}, $G$ is a
$1$-K\"{o}nig-Egerv\'{a}ry graph. For every $v\in V\left(  G\right)  $ we have%
\[
\alpha\left(  G\right)  +\mu\left(  G\right)  =n\left(  G\right)
-1=\alpha\left(  G-v\right)  +\mu\left(  G-v\right)  \text{,}%
\]
which, by Lemma \ref{lem1}, means that $\alpha\left(  G\right)  =\alpha\left(
G-v\right)  $ and $\mu\left(  G\right)  =\mu\left(  G-v\right)  $. In other
words, $G$ has neither $\alpha$-critical vertices ($\emph{core}\left(
G\right)  =\emptyset$) nor perfect matchings. In may be put down like
$\xi\left(  G\right)  =0$ and $\mu\left(  G\right)  <\frac{n\left(  G\right)
}{2}$.

By Theorem \ref{th3} and Theorem \ref{cor18}, we know that
\[
\varrho_{v}\left(  G\right)  \leq n\left(  G\right)  +d\left(  G\right)
-\xi\left(  G\right)  -\beta(G).
\]
Hence,
\begin{align*}
0  &  \leq d\left(  G\right)  -\beta(G)\Longleftrightarrow\\
\left\vert \emph{diadem}\left(  G\right)  \right\vert  &  \leq\left\vert
\emph{diadem}\left(  G\right)  \right\vert -\left\vert N\left(  \emph{diadem}%
\left(  G\right)  \right)  \right\vert \Longleftrightarrow\\
\left\vert N\left(  \emph{diadem}\left(  G\right)  \right)  \right\vert  &
=0\Longleftrightarrow\left\vert \emph{diadem}\left(  G\right)  \right\vert
=0\Longleftrightarrow\emph{diadem}\left(  G\right)  =\emptyset,
\end{align*}
which means $\beta\left(  G\right)  =0$ and, consequently, completes the proof.

\emph{"Only If" Part.} Suppose that $G$ is a $1$-K\"{o}nig-Egerv\'{a}ry graph
(i.e., $\alpha\left(  G\right)  +\mu\left(  G\right)  =n\left(  G\right)  -1$)
without perfect matchings (i.e., $\mu\left(  G\right)  <\frac{n\left(
G\right)  }{2}$), $\xi\left(  G\right)  =0$ and $\beta\left(  G\right)  =0$.
We have to prove that $\alpha\left(  G-v\right)  +\mu\left(  G-v\right)
=n\left(  G-v\right)  $ for each $v\in V\left(  G\right)  $.

Since $\xi\left(  G\right)  =0$, there are no $\alpha$-critical vertices,
i.e., $\alpha\left(  G\right)  =\alpha\left(  G-v\right)  $ for every $v\in
V\left(  G\right)  $. Hence, it is enough to prove that $\mu\left(  G\right)
=\mu\left(  G-v\right)  $ for every $v\in V\left(  G\right)  $.

The condition that $\beta\left(  G\right)  =0$ is equivalent to the fact that
$G$ has only one critical set, namely, the empty set. In other words, $G$ is
\textit{regularizable}, i.e., $\left\vert A\right\vert <\left\vert
N_{G}\left(  A\right)  \right\vert $ for every non-empty independent set $A$
of $G$ \cite{Berge1982}. In particular, if $S\in\Omega\left(  G\right)  $,
then
\[
\left\vert S\right\vert <\left\vert N_{G}\left(  S\right)  \right\vert
=\left\vert V\right\vert -\left\vert S\right\vert \Longleftrightarrow
\alpha\left(  G\right)  <\frac{n\left(  G\right)  }{2}.
\]

\emph{Claim 1.} $\alpha\left(  G\right)  =\mu\left(  G\right)  $.

Theorem \ref{th12}\emph{(i),(ii)} imply $\mu(G)\leq\alpha(G)$, because $G$ has
no perfect matchings. On the other hand, if $S\in\Omega\left(  G\right)  $,
then by Hall's Theorem, there is a matching from $S$ into $V\left(  G\right)
-S$, because $G$ is regularizable. Therefore, $\alpha(G)\leq\mu(G)$. Thus
$\alpha\left(  G\right)  =\mu\left(  G\right)  $.

\emph{Claim 2.} $\mu\left(  G\right)  =\mu\left(  G-v\right)  $ for each $v\in
V\left(  G\right)  $.

Let $v\in V\left(  G\right)  $. Hence, there exists a maximum independent set
$S$, such that $v\notin S$, because $\emph{core}\left(  G\right)  =\emptyset$.
As mentioned above, $\left\vert A\right\vert <\left\vert N_{G}\left(
A\right)  \right\vert $ for every $A\in\emph{Ind}\left(  G\right)  $.
Therefore, $\left\vert A\right\vert \leq\left\vert N_{G-v}\left(  A\right)
\right\vert $ for every independent set $A\subseteq S$. Consequently, by
Hall's Theorem, there is a matching, say $M$, from $S$ into $V\left(
G\right)  -\left\{  v\right\}  -S$. Clearly, by Lemma \ref{lem1}\emph{(iii)}
and \emph{Claim 1},
\[
\mu\left(  G\right)  \geq\mu\left(  G-v\right)  \geq\left\vert M\right\vert
=\left\vert S\right\vert =\alpha\left(  G\right)  =\mu\left(  G\right)  .
\]

Finally, $\mu\left(  G-v\right)  =\mu\left(  G\right)  $, which completes the
whole proof.
\end{proof}

For example, the Friendship graph $F_{k}$ is a $1$-K\"{o}nig-Egerv\'{a}ry
graph, $\emph{core}\left(  F_{k}\right)  =\emptyset$, $\emph{diadem}\left(
G\right)  =\emptyset$, and, by Theorem \ref{lem10}, $F_{k}$ is critical vertex
almost K\"{o}nig-Egerv\'{a}ry.

\begin{remark}
Claim 1 of Theorem \ref{lem10} mentions that if $G$ is \textit{critical vertex
almost K\"{o}nig-Egerv\'{a}ry, then} $\alpha\left(  G\right)  =\mu\left(
G\right)  $. Since $G$ is $1$-K\"{o}nig-Egerv\'{a}ry, it follows that
\[
n\left(  G\right)  =\alpha\left(  G\right)  +\mu\left(  G\right)
+1=2\alpha\left(  G\right)  +1,
\]
i.e., the order of $G$\ is odd.
\end{remark}

The following examples point out to the fact that conditions of Theorem
\ref{lem10} can not be weakened.

\begin{itemize}
\item There are $1$-K\"{o}nig-Egerv\'{a}ry graphs with a perfect matching
having $\emph{diadem}\left(  G\right)  =\emptyset$ and $\emph{core}\left(
G\right)  =\emptyset$, that are not critical vertex almost
K\"{o}nig-Egerv\'{a}ry, for instance $K_{4}$.

\item The graph $K_{5}$ shows that being a graph without perfect matchings,
$\emph{core}\left(  G\right)  =\emptyset$, $\emph{diadem}\left(  G\right)
=\emptyset$ is not enough to be critical vertex almost K\"{o}nig-Egerv\'{a}ry.

\item The $1$-K\"{o}nig-Egerv\'{a}ry graph $G_{2}$ in Figure \ref{fig111} has
$\alpha\left(  G_{2}\right)  =\mu\left(  G_{2}\right)  $, and $\emph{diadem}%
\left(  G\right)  =\emptyset$, while $\emph{core}\left(  G\right)
\neq\emptyset$, and, consequently, $G_{2}$ is not critical vertex almost
K\"{o}nig-Egerv\'{a}ry\textit{.}

\item The $1$-K\"{o}nig-Egerv\'{a}ry graph $G_{2}$ in Figure \ref{fig222} has
$\alpha\left(  G_{2}\right)  =\mu\left(  G_{2}\right)  $, and $\emph{core}%
\left(  G\right)  =\emptyset$, while $\emph{diadem}\left(  G\right)
\neq\emptyset$, and, consequently, $G_{2}$ is not critical vertex almost
K\"{o}nig-Egerv\'{a}ry\textit{.}
\end{itemize}

Using Theorem \ref{lem10} and the corresponding theorem from \cite{LevMan2024}%
, we may conclude with the following.

\begin{theorem}
Let $G\ $be a graph of order $n\left(  G\right)  \geq1$. Then $\varrho
_{v}\left(  G\right)  =n\left(  G\right)  $ if and only if either

\emph{(i)} $G$ is K\"{o}nig-Egerv\'{a}ry and $\emph{core}\left(  G\right)
=\ker(G)$;

or

\emph{(ii) }$G$ is $1$-K\"{o}nig-Egerv\'{a}ry without perfect matchings,
$\xi\left(  G\right)  =0$ and $\beta\left(  G\right)  =0$.
\end{theorem}

It is worth mentioning that if a graph $G$ is neither K\"{o}nig-Egerv\'{a}ry
nor $1$-K\"{o}nig-Egerv\'{a}ry, then $\varrho_{v}\left(  G\right)  =0$ in
accordance with Theorem \ref{th8}.

\section{Conclusions}

\begin{theorem}
\label{th10}If there is an edge $e\in E\left(  G\right)  $, such that $G-e$ is
a K\"{o}nig-Egerv\'{a}ry graph, then $G$ is either K\"{o}nig-Egerv\'{a}ry or
$1$-K\"{o}nig-Egerv\'{a}ry.
\end{theorem}

\begin{proof}
By definition, we know that
\[
\alpha(G-e)+\mu(G-e)=n\left(  G-e\right)  =n\left(  G\right)  .
\]

\textit{Case 1.} $\alpha(G-e)=\alpha(G)$.

Thus $\alpha(G)+\mu(G-e)=n\left(  G\right)  $. If $\mu(G-e)=\mu(G)$, then
$\alpha(G)+\mu(G)=n\left(  G\right)  $, which means that $G$ is a
K\"{o}nig-Egerv\'{a}ry graph; otherwise, $\mu(G-e)=\mu(G)-1$ and,
consequently, $\alpha(G)+\mu(G)=n\left(  G\right)  -1$, i.e., $G$ is a
$1$-K\"{o}nig-Egerv\'{a}ry graph.

\textit{Case 2.} $\alpha(G-e)=\alpha(G)+1$.

Thus $\alpha(G)+1+\mu(G-e)=n\left(  G\right)  $. If $\mu(G-e)=\mu(G)$, then
$\alpha(G)+\mu(G)=n\left(  G\right)  -1$, which means that $G$ is a
$1$-K\"{o}nig-Egerv\'{a}ry graph; otherwise, $\mu(G-e)=\mu(G)-1$ and,
consequently, $\alpha(G)+\mu(G)=n\left(  G\right)  $, i.e., $G$ is a
K\"{o}nig-Egerv\'{a}ry graph.
\end{proof}

By Theorem \ref{th10}, if $G$ is neither K\"{o}nig-Egerv\'{a}ry nor
$1$-K\"{o}nig-Egerv\'{a}ry, then $\varrho_{e}\left(  G\right)  =0$. Theorem
\ref{th9} claims that if $G$ is K\"{o}nig-Egerv\'{a}ry, then $\varrho
_{e}\left(  G\right)  \leq m\left(  G\right)  -\xi\left(  G\right)
+\varepsilon\left(  G\right)  $. It justifies the following.

\begin{problem}
Bound $\varrho_{e}\left(  G\right)  $ using various graph invariants for
$1$-K\"{o}nig-Egerv\'{a}ry graphs.
\end{problem}

A graph $G$ is a \textit{critical edge almost K\"{o}nig-Egerv\'{a}ry graph},
if $G$ is not a K\"{o}nig-Egerv\'{a}ry graph, but $G-e$ is a
K\"{o}nig-Egerv\'{a}ry graph for every $e\in E(G)$, i.e., $\varrho_{e}\left(
G\right)  =m\left(  G\right)  $. For instance, every $C_{2k+1}$ is a critical
edge almost K\"{o}nig-Egerv\'{a}ry graph. It is worth reminding that, by
Theorem \ref{th3}, $G$ is $1$-K\"{o}nig-Egerv\'{a}ry.

\begin{lemma}
If $G$ is a \textit{critical edge almost K\"{o}nig-Egerv\'{a}ry graph, then
every edge is }$\alpha$-critical.
\end{lemma}

\begin{proof}
For each $e\in E\left(  G\right)  $, we have%
\[
\alpha\left(  G\right)  +\mu\left(  G\right)  \leq\alpha\left(  G-e\right)
+\mu\left(  G\right)  \leq n\left(  G\right)  =\alpha\left(  G-e\right)
+\mu\left(  G-e\right)  \text{,}%
\]
which implies that $\mu\left(  G\right)  \leq\mu\left(  G-e\right)  $.
Therefore, $\mu\left(  G-e\right)  =\mu\left(  G\right)  $ and further,
\[
\alpha\left(  G\right)  +\mu\left(  G\right)  <n\left(  G\right)
=\alpha\left(  G-e\right)  +\mu\left(  G\right)
\]
ensures that $e\in E(G)$ must be $\alpha$-critical.
\end{proof}

The converse is not true; e.g., $K_{n},n\geq5$.

\begin{problem}
\label{problem1}Characterize \textit{critical edge almost
K\"{o}nig-Egerv\'{a}ry graphs.}
\end{problem}

Theorem \ref{th9} claims that if $G$ is K\"{o}nig-Egerv\'{a}ry, then
$\varrho_{v}\left(  G\right)  =n\left(  G\right)  -\xi\left(  G\right)
+\varepsilon\left(  G\right)  $. It justifies the following.

\begin{problem}
Express $\varrho_{v}\left(  G\right)  $ using various graph invariants for
$1$-K\"{o}nig-Egerv\'{a}ry graphs.
\end{problem}

\section{Declarations}

\textbf{Conflict of interest} We declare that we have no conflict of interest.

\end{document}